\documentclass[12 pt, reqno]{amsart}
\usepackage{mathrsfs}
\usepackage{graphicx}
\usepackage{tikz}
\usepackage{amsmath}
\usepackage{amsfonts}
\usepackage{amssymb}
\usepackage{hyperref}

\newtheorem{theorem}{Theorem}
\newtheorem{lemma}[theorem]{Lemma}

\newtheorem{corollary}[theorem]{Corollary}

\theoremstyle{definition}
\newtheorem{remark}{Remark}

\usepackage{amsthm,amsmath,amssymb,url,cite,color}

\numberwithin{equation}{section}

\def\I{\mathfrak{I}}
 \def\P{\mathcal{P}}
\def\asc{\operatorname{asc}}
\def\Asc{\operatorname{Asc}}
\def\Des{\operatorname{DES}}
\def\wt{\operatorname{wt}}
\def\PF{\operatorname{PF}}

\def\e{{\bf e}}
\def\s{{\bf s}}
\def\mm{{\bf m}}

\def\des{\operatorname{des}}
\def\maj{\operatorname{maj}}
\def\fmaj{\operatorname{fmaj}}

\def\fdes{\operatorname{fdes}}

\def\N{\mathbb N}
\def\Z{\mathbb Z}
\def\R{\mathbb R}
\def\PP{\mathbb P}

\begin{document}

\title[On the descent polynomial of signed multipermutations]{On the descent polynomial of signed multipermutations}

\author{Zhicong Lin}
\address[Zhicong Lin]{Institut Camille Jordan, UMR 5208 du CNRS, Universit\'{e} Claude Bernard Lyon 1, France}
\email{lin@math.univ-lyon1.fr}

\date{\today}

\begin{abstract}
 Motivated by a  conjecture of Savage and Visontai about the equidistribution of the descent statistic on signed permutations of the multiset $\{1,1,2,2,\ldots,n,n\}$ and the ascent statistic on $(1,4,3,8,\ldots,2n-1,4n)$-inversion sequences,  we investigate the descent polynomial of the signed permutations of a general multiset. We obtain a factorial generating function formula for a $q$-analog of these descent polynomials and apply it to show that they have only real roots.
 Two different proofs of the conjecture of Savage and Visontai are provided. 
\end{abstract}

\maketitle

 %%%%%%%%%%%%%%%%%%%%%%
\section{Introduction}
%%%%%%%%%%%%%%%%%%%%%
For a sequence of positive integers $\s=\{s_i\}_{i\geq1}$, let the set of {\em$\s$-inversion sequences of length $n$}, $\I_n^{\s}$, be defined as
$$
\I_n^{(\s)}:=\{(e_1,\ldots,e_n)\in\Z^n : 0\leq e_i<s_i\,\text{for}\, 1\leq i\leq n\}.
$$
The {\em ascent set} of an $\s$-inversion sequence $\e=(e_1,\ldots,e_n)\in\I_n^{\s}$ is the set 
$$
\Asc(\e):=\left\{0\leq i<n : \frac{e_i}{s_i}<\frac{e_{i+1}}{s_{i+1}}\right\},
$$
with the convention that $e_0=0$ and $s_0=1$. Let $\asc(\e):=|\Asc(\e)|$ be the {\em ascent statistic} on $\e\in\I_n^{\s}$. The $\s$-inversion sequences and the ascent statistic were introduced by  Savage and Schuster in~\cite{ss}.

A {\em descent} in a permutation $\pi=\pi_1\pi_2\cdots\pi_n$ of a multiset with elements from $\N$ is an index $i\in\{0,1,\ldots,n-1\}$ such that $\pi_i>\pi_{i+1}$ (with the convention that $\pi_0=0$). Denote by $\Des(\pi)$ the set of descents of $\pi$ and  by $\des(\pi)$ the number of descents of $\pi$. The major index of $\pi$, denoted $\maj(\pi)$, is defined as 
$$
\maj(\pi):=\sum_{i\in\Des(\pi)}i.
$$

Let $P(\{1,1,2,2,\ldots,n,n\})$ be the set of permutations of the multiset $\{1,1,2,2,\ldots,n,n\}$.
The following  connection between multiset permutations and  special inversion sequences was shown in~\cite[Theorem~3.23]{sv}.

\begin{theorem}[Savage-Visontai~\cite{sv}]
$$
\sum_{\pi\in P(\{1,1,2,2,\ldots,n,n\})} t^{\des(\pi)}=\sum_{\e\in\I_{2n}^{(1,1,3,2,\ldots,2n-1,n)}}t^{\asc(\e)}
$$
\end{theorem}

Now consider the signed multiset permutations.
Let $P^{\pm}(\{1,1,2,2,\ldots,n,n\})$ be the set of all {\em signed} permutations of the multiset $\{1,1,2,2,\ldots,n,n\}$, whose elements are those of the form $\pm\pi_1\pm\pi_2\cdots\pm\pi_n$ with $\pi_1\pi_2\cdots\pi_n\in P(\{1,1,2,2,\ldots,n,n\})$. For convenient, we write $-n$ by $\overline{n}$ for each positive integer $n$. For example, 
$$
P^{\pm}(\{1,1\})=\{1\,1,1\,\overline{1},\overline{1}\,1,\overline{1}\,\overline{1}\}.
$$
Clearly, we have 
$$
|P^{\pm}(\{1,1,2,2,\ldots,n,n\})|=\frac{(2n)!}{2^n}2^{2n}=2^n(2n)!=|\I_{2n}^{(1,4,3,8,\ldots,2n-1,4n)}|.
$$
Savage and Visontai~\cite[Conjecture~3.25]{sv} further conjectured the following equdistribution, which was proved very recently (and independently) by Chen et al.~\cite{ch} using type B $P$-Partitions.
\begin{theorem}\label{conj:VS}
For any $n\geq1$, we have
\begin{equation}\label{con:parti}
\sum_{\pi\in P^{\pm}(\{1,1,2,2,\ldots,n,n\})} t^{\des(\pi)}=\sum_{\e\in\I_{2n}^{(1,4,3,8,\ldots,2n-1,4n)}}t^{\asc(\e)}.
\end{equation}
\end{theorem}
For $n\leq2$, we have 
$$
\sum_{\pi\in P^{\pm}(\{1,1\})} t^{\des(\pi)}=1+3t=\sum_{\e\in\I_{2}^{(1,4)}}t^{\asc(\e)}
$$
and 
$$
\sum_{\pi\in P^{\pm}(\{1,1,2,2\})} t^{\des(\pi)}=1+31t+55t^2+9t^3=\sum_{\e\in\I_{4}^{(1,4,3,8)}}t^{\asc(\e)}.
$$

In section~\ref{conj:sv}, we will give a simple proof of Theorem~\ref{conj:VS} through verifying the recurrence formulas for both sides of Eq.~\eqref{con:parti}.   Motivated by this conjecture, we study the descent polynomial of signed permutations of a general multiset (or called general signed multipermutations for short). In section~\ref{gen:chow}, we proved a factorial generating function formula for the $(\des,\fmaj)$-enumerator of general signed multipermutations, which generalizes a result of Chow and Gessel~\cite{cg}. Finally, in section~\ref{sig:simion}, we use this factorial formula to show that the descent polynomial of the signed multipermutations is real-rootedness.

\section{Proof of Theorem~\ref{conj:VS}}
\label{conj:sv}

\begin{lemma}\label{lem:1}
Let 
$$
E_n(t)=\sum_{i=0}^{2n-1}E_{n,i}t^i:=\sum_{\e\in\I_{2n}^{(1,4,3,8,\ldots,2n-1,4n)}}t^{\asc(\e)}.
$$ 
Then 
\begin{equation}\label{rec:Ent}
E_{n+1,i}=(2i^2+3i+1)E_{n,i}+(2i(4n-2i+3)+2n+1)E_{n,i-1}+(2n+2-i)(4n-2i+5)E_{n,i-2},
\end{equation}
with boundary conditions $E_{n,i}=0$ for $i<0$ or $i>2n-1$.
\end{lemma}
\begin{proof} The following formula was established in~\cite[Theorem~13]{ss} using Ehrhart theory:
\begin{equation}\label{eq:ehrhart}
\frac{E_n(t)}{(1-t)^{2n+1}}=\sum_{k\geq0}((k+1)(2k+1))^nt^k.
\end{equation}
Thus we have
\begin{align*}
\frac{E_{n+1}(t)}{(1-t)^{2n+3}}&=\sum_{k\geq0}((k+1)(2k+1))^{n+1}t^k\\
&=\sum_{k\geq0}((k+1)(2k+1))^{n}(k+1)(2k+1)t^k\\
&=\sum_{k\geq0}((k+1)(2k+1))^{n}(2k(k-1)+5k+1)t^k\\
&=2t^2(E_{n}(t)(1-t)^{-2n-1})''+5t(E_{n}(t)(1-t)^{-2n-1})'+\frac{E_{n}(t)}{(1-t)^{2n+1}}.
\end{align*}
The recursive formula~\eqref{rec:Ent} then follows from the above equation by multiplying both sides by $(1-t)^{2n+3}$ and then extracting the coefficients of $t^i$.
\end{proof}

\begin{lemma}\label{lem:2}
Let 
$$
P_n(t)=\sum_{i=0}^{2n-1}P_{n,i}t^i:=\sum_{\pi\in P^{\pm}(\{1,1,2,2,\ldots,n,n\})} t^{\des(\pi)}.
$$
Then 
\begin{equation}\label{rec:Pni}
P_{n+1,i}=(2i^2+3i+1)P_{n,i}+(2i(4n-2i+3)+2n+1)P_{n,i-1}+(2n+2-i)(4n-2i+5)P_{n,i-2},
\end{equation}
with boundary conditions $P_{n,i}=0$ for $i<0$ or $i>2n-1$.
\end{lemma}

\begin{proof}
Denote by $\P_{n,i}$ the set of signed permutation of $\{1,1,2,2,\ldots,n,n\}$ with $i$ descents.
Clearly, every signed permutation in $\P_{n+1,i}$ can be obtained from one of the following three different cases:
\begin{enumerate}
\item from a signed permutation in $\P_{n,i}$ by inserting $\{n+1,n+1\}$, $\{\overline{n+1},\overline{n+1}\}$ or $\{n+1,\overline{n+1}\}$;
\item from a signed permutation in $\P_{n,i-1}$ by inserting $\{n+1,n+1\}$, $\{\overline{n+1},\overline{n+1}\}$ or $\{n+1,\overline{n+1}\}$;
\item from a signed permutation in $\P_{n,i-2}$ by inserting $\{n+1,n+1\}$, $\{\overline{n+1},\overline{n+1}\}$ or $\{n+1,\overline{n+1}\}$.
\end{enumerate}

In case $(1)$,
% for each $\pi\in\P_{n,i}$, we can insert 
%\begin{enumerate}
%\item two $n+1$ to one of the $i$ descent slots or to the right of $\pi$;
%\item $n+1$ 
%\end{enumerate}
one can check that there are $2i^2+3i+1$ ways to insert $\{n+1,n+1\}$, $\{\overline{n+1},\overline{n+1}\}$ or $\{n+1,\overline{n+1}\}$ into each signed permutation in $\P_{n,i}$ to become a signed permutation in $\P_{n+1,i}$ (be careful when $n+1$ or $\overline{n+1}$ is inserted to the right of a signed permutation). So there are $(2i^2+3i+1)P_{n,i}$ signed permutations constructed from case $(1)$. Similarly, there are $(2i(4n-2i+3)+2n+1)P_{n,i-1}$ and $(2n+2-i)(4n-2i+5)P_{n,i-2}$ signed permutations in $\P_{n+1,i}$ constructed from cases $(2)$ and $(3)$, respectively. Those amount to the right hand side of~\eqref{rec:Pni}, which completes the proof.
\end{proof}

By Lemma~\ref{lem:1} and~\ref{lem:2}, we see that $E_{n,i}$ and $P_{n,i}$ satisfy the same recurrence relation and boundary conditions, so they are equal. This finishes the proof of Theorem~\ref{conj:VS}.

\begin{remark}
We do not see how to deduce recurrence~\eqref{rec:Ent} for $E_{n,i}$ directly from its interpretation as inversion sequences.
\end{remark}

\section{General signed multipermutations}
\label{gen:chow}

In this section, we consider the descent polynomial on signed permutations of the general multiset 
$
M_{\mm}:=\{1^{m_1},2^{m_2},\ldots,n^{m_n}\}
$
for each vector $\mm:=(m_1,m_2,\ldots,m_n)\in\PP^n$. Let $P(\mm)$ and $P^{\pm}(\mm)$ denote the set of all permutations and signed permutations of the multiset $M_{\mm}$, respectively.

Recall that the $q$-shift factorial $(t;q)_n$ is defined by $(t;q)_n :=\prod_{i=0}^{n-1}(1-tq^i)$ for any positive integer $n$ and $(t;q)_0=1$. The $q$-binomial coefficient ${n\brack k}_q$ is then defined as 
$$
{n\brack k}_q:=\frac{(q;q)_n}{(q;q)_{n-k}(q;q)_k}.
$$
The following  is a $q$-analog of a result of MacMahon~\cite[Volme~2, p.~211]{MM}, whose proof can  be found in~\cite[\S~7]{fh:q}.

\begin{theorem} \label{macmahon}
For every $\mm\in\PP^n$ with $m_1+\cdots+m_n=m$, we have
\begin{equation*}
\frac{\sum_{\pi\in P(\mm)}t^{\des(\pi)}q^{\maj(\pi)}}{(t;q)_{m+1}}=\sum_{k\geq0}{m_1+k\brack m_1}_q\cdots{m_n+k\brack m_n}_qt^k.
\end{equation*}
\end{theorem}

Our signed version is:

\begin{theorem} \label{thm:signed}
For every $\mm\in\PP^n$ with $m_1+\cdots+m_n=m$, we have
\begin{equation}\label{sign:eq}
\frac{\sum_{\pi\in P^{\pm}(\mm)}t^{\des(\pi)}q^{\fmaj(\pi)}z^{N(\pi)}}{(t;q^2)_{m+1}}=\sum_{k\geq0}\prod_{r=1}^n\left(\sum_{i=0}^{m_r}(zq)^i{m_r-i+k\brack m_r-i}_{q^2}{i+k-1\brack i}_{q^2}\right)t^k,
\end{equation}
where $\fmaj(\pi):=2\maj(\pi)+N(\pi)$ and $N(\pi)$ is the number of negative signs in $\pi$.
\end{theorem}
\begin{remark}
Setting $z=0$, we recover Theorem~\ref{macmahon}.
It is worth noticing that Foata and Han~\cite[Theorem~1.2]{fh:sw3} has calculated (another signed version of Theorem~\ref{macmahon}) the factorial generating function formula for 
$$
\sum_{\pi\in P^{\pm}(\mm)}t^{\fdes(\pi)}q^{\fmaj(\pi)}z^{N(\pi)},
$$
involving  the so-called {\em flag descent statistic} $\fdes$, $\fdes(\pi):=2\des(\pi)-\chi(\pi_1<0)$, on signed multipermutations (or words).
\end{remark}

\begin{corollary}[Chow-Gessel~\cite{cg}]
$$
\frac{\sum_{\pi\in P^{\pm}(\{1,2,\ldots,n\})} t^{\des(\pi)}q^{\fmaj(\pi)}}{(t;q^2)_{2n+1}}=\sum_{k\geq0}[2k+1]_q^nt^k.
$$
\end{corollary}
\begin{proof}
Setting $m_1=\cdots=m_n=1$ and $z=1$ in Theorem~\ref{thm:signed}.
\end{proof}

\begin{corollary}
\begin{equation}\label{eq:sign}
\frac{\sum_{\pi\in P^{\pm}(\{1,1,2,2,\ldots,n,n\})} t^{\des(\pi)}}{(1-t)^{2n+1}}=\sum_{k\geq0}((k+1)(2k+1))^nt^k.
\end{equation}
\end{corollary}
\begin{proof}
Setting $m_1=\cdots=m_n=2$ and $z=q=1$ in Theorem~\ref{thm:signed}.
\end{proof}
\begin{remark}
Comparing~\eqref{eq:sign} with~\eqref{eq:ehrhart} we get another proof of Theorem~\ref{conj:VS}. As was already noticed in~\cite{sv}, there is not natural  $q$-analog of Eq.~\eqref{eq:ehrhart}.
\end{remark}

We will prove Theorem~\ref{thm:signed} by using the technique of barred permutation motivated by Gessel and Stanley~\cite{gs}. For each $\pi=\pi_1\cdots\pi_m\in P^{\pm}(\mm)$, we call the space between $\pi_i$ and $\pi_{i+1}$ the  $i$-th space of $\pi$ for $0< i< m$. We also call the space before $\pi_1$ and the space after $\pi_{m}$ the $0$-th space and the $m$-th space of $\pi$, respectively. If $i\in\Des(\pi)$, then we call the $i$-th space a descent space.

A {\em barred permutation} on $\pi\in P^{\pm}(\mm)$ is obtained by  inserting one or more vertical bars into some spaces of $\pi$  such that there is at least one bar in every descent space of $\pi$.  
 For example, $||\bar{1}|\bar{2}2|1||$ is a barred permutation on $\pi=\bar{1}\bar{2}21$ but $\bar{1}|\bar{2}2|1||$ is not, since $0\in\Des(\pi)$ and there is not bar in the $0$-th space (i.e. before $\pi_1=\bar{1}$).

\begin{proof}[{\bf Proof of Theorem~\ref{thm:signed}}] Let $B(\mm)$ be the set of barred permutations on $P^{\pm}(\mm)$. Let $\sigma\in B(\mm)$ be a barred permutation on $\pi$ with $b_i$ bars in the $i$-th space of $\pi$, we define the weight  $\wt(\sigma)$ to be
$$
\wt(\sigma):=t^{\sum b_i}q^{N(\pi)+2\sum i b_i}z^{N(\pi)}.
$$
For example, $\wt(||\bar{1}|\bar{2}2|1||)=t^6q^{26}z^2$. Now we count the barred permutations in $B(\mm)$ by the  weight $\wt$ in two different ways. First, fix a permutation $\pi$, and sum over all barred permutations on $\pi$. Then fix the number of bars $k$, and sum over all barred permutations with $k$ bars.

Fix a permutation $\pi$, a barred permutation on $\pi$ can be obtained by inserting one bar in each descent space and then inserting any number of bars in every space. So counting all the barred permutations on $\pi$ by the weight $\wt$ gives
$$
t^{\des(\pi)}q^{\fmaj(\pi)}z^{N(\pi)}(1+t+t^2+\cdots)(1+tq^2+(tq^2)^2+\cdots)\cdots(1+tq^{2m}+(tq^{2m})^2+\cdots),
$$
which is equal to $\frac{t^{\des(\pi)}q^{\fmaj(\pi)}z^{N(\pi)}}{(t;q^2)_{m+1}}$. This shows that 
\begin{equation}\label{left}
\sum_{\sigma\in B(\mm)}\wt(\sigma)=\frac{\sum_{\pi\in P^{\pm}(\mm)}t^{\des(\pi)}q^{\fmaj(\pi)}z^{N(\pi)}}{(t;q^2)_{m+1}}.
\end{equation}

For a fix $k\geq0$, let $B_k(\mm)$ be the set of barred permutations in $B(\mm)$ with $k$ bars. Now each barred permutation in $B_k(\mm)$ can be constructed by putting $k$ bars in one line and then inserting $m_r$ integers from $\{r,-r\}$, for $1\leq r\leq n$, to the $k+1$ spaces between each two adjacency bars (including the space in the left side and the right side), with the rule that  negative integers can not be inserted to the left side. Thus by the well-known interpretation (cf.~\cite[Proposition~4.1]{fh:q}) of the $q$-binomial coefficient 
$$
{n+r\brack n }_q=\sum_{0\leq c_1\leq c_2\cdots\leq c_n\leq r}q^{\sum c_i},
$$
we have 
\begin{equation*}
\sum_{\sigma\in B_k(\mm)} \wt(\sigma)=\prod_{r=1}^n\left(\sum_{i=0}^{m_r}(zq)^i{m_r-i+k\brack m_r-i}_{q^2}{i+k-1\brack i}_{q^2}\right)t^k.
\end{equation*}
Summing over all $k$ in the above equation and comparing with Eq.~\eqref{left} we get~\eqref{sign:eq}.
\end{proof}

%\subsection{Some other equidistributions}
%In~\cite{ch}, Chen et al. discovered three further equidistributions of the descent statistic on some special signed multipermutations and the ascent statistic on other inversion sequences. We will show that there is still one missing equidistribution. 
%
%Let $V^-_n$ be the subset of the signed permutations in $P^{\pm}(\{1^2,2^2,\ldots,(n-1)^2,n\})$ such that the element $n$ carries a minus sign and let $V^+_n$ be the complement of $V^-_n$ with respect to $P^{\pm}(\{1^2,2^2,\ldots,(n-1)^2,n\})$. 
%\begin{theorem}[Chen et al.~\cite{ch}]
%For $n\geq1$, we have the following three equidistributions
%\begin{align*}
%\sum_{\pi\in P^{\pm}(\{1^2,2^2,\ldots,(n-1)^2,n^2\})} t^{\des(\pi)}&=\sum_{\e\in\I_{2n}^{(2,2,6,4,\ldots,4n-2,2n)}}t^{\asc(\e)},\\
%\sum_{\pi\in P^{\pm}(\{1^2,2^2,\ldots,(n-1)^2,n\})} t^{\des(\pi)}&=\sum_{\e\in\I_{2n-1}^{(2,2,6,4,\ldots,4n-6,2(n-1),4n-2)}}t^{\asc(\e)},\\
%\sum_{\pi\in V^+_n} t^{\des(\pi)}&=\sum_{\e\in\I_{2n-1}^{(1,4,3,8,\ldots,2n-3,4(n-1),2n-1)}}t^{\asc(\e)}.
%\end{align*}
%\end{theorem}
%
%The following equidistribution is a complement of the last equation above.
%\begin{theorem}
%For $n\geq1$, we have 
%\begin{equation}
%\sum_{\pi\in V^+_n} t^{\des(\pi)}=\sum_{\e\in\I_{2n-1}^{(1,4,3,8,\ldots,2n-3,4(n-1),2n-1)}}t^{\asc(\e)}.
%\end{equation}
%\end{theorem}
%

\section{Signed version of Simion's result about real-rootedness}
\label{sig:simion}

It was shown in~\cite[Theorem~1.1]{sv} that the ascent polynomial $\sum_{\e\in\I_{2n}^{(\s)}}t^{\asc(\e)}$ has only real roots for each $\s\in\PP^n$. In view of Theorem~\ref{conj:VS} we have
\begin{corollary}\label{real:spec2}
The polynomial
$$
\sum_{\pi\in P^{\pm}(\{1,1,2,2,\ldots,n,n\})} t^{\des(\pi)}
$$
has only real roots for any positive integer $n$.
\end{corollary}

Simion~\cite[\S~2]{si} proved that the descent polynomial on the permutations of a general multiset has only real roots. We have the following signed version of Simion's result, which generalizes the $m_1=m_2=\cdots=m_n=1$ (i.e. the type B coxeter group) case of Brenti~\cite{br} and Corollary~\ref{real:spec2}.

\begin{theorem}\label{realroot}
The descent polynomial 
$$
\sum_{\pi\in P^{\pm}(\mm)}t^{\des(\pi)}
$$
has only real roots for every $\mm\in\PP^n$. 
\end{theorem}

The key point of the proof of the above result lies in the following simple lemma. Let $\PF_1[t^i]$ be the set of all polynomials with nonnegative coefficients and let $\PF[t^i]$ be the set of all real-rooted polynomials in $\PF_1[t^i]$. Actually, $\PF[t^i]$ is the set of all polynomials in $\R[t]$ whose coefficients form a {\em P\'oly frequency sequence} (cf.~\cite[Theorem~2.2.4]{br}).

\begin{lemma}\label{euler}
Let 
\begin{equation}\label{rec:ab}
\frac{F_{n}(t)}{(1-t)^{n+1}}=\sum_{k\geq0}f(k)(ak+b)t^k\quad\text{and}\quad \frac{F_{n-1}(t)}{(1-t)^{n}}=\sum_{k\geq0}f(k)t^k.
\end{equation}
If $a>0, n>\frac{b}{a}$, $F_{n-1}(t)\in \PF[t^i]$ and $F_{n}(t)\in \PF_1[t^i]$, then we have $F_{n}(t)\in \PF[t^i]$.
\end{lemma}
\begin{proof} Clearly, by~\eqref{rec:ab} we have
\begin{equation*}
\frac{F_{n}(t)}{(1-t)^{n+1}}=\sum_{k\geq0}f(k)(ak+b)t^k=b\frac{F_{n-1}(t)}{(1-t)^{n}}+a\left(\frac{F_{n-1}(t)}{(1-t)^{n}}\right)'t. 
\end{equation*}
From the above equation we deduce that
$$
F_{n+1}(t)=\left((an-b)t+b\right)F_n(t)+at(1-t)F_n'(t).
$$
The lemma then follows by applying a result of Brenti~\cite[Theorem~2.4.5]{br}, which was established through some standard argument by using Rolle's theorem.
\end{proof}

\begin{proof}[{\bf Proof of Theorem~\ref{realroot}}]
By setting $q=1,z=1$ in Eq.~\eqref{sign:eq} and using Binomial theorem, we have 
\begin{equation}\label{sig:nq}
\frac{\sum_{\pi\in P^{\pm}(\mm)}t^{\des(\pi)}}{(1-t)^{m+1}}=\sum_{k\geq0}\prod_{r=1}^n\frac{(2k+1)(2k+2)\cdots(2k+m_r)}{m_r!}t^k.
\end{equation}
Observe that for any permutation $\pi$ of $[n]$, 
$$
\sum_{\pi\in P^{\pm}(\mm)}t^{\des(\pi)}=\sum_{\pi\in P^{\pm}(\pi(\mm))}t^{\des(\pi)},
$$
where $\pi(\mm)=\{1^{m_{\pi(1)}},2^{m_{\pi(2)}},\ldots,n^{m_{\pi(n)}}\}$. So we can assume that $m_1\geq m_2\geq\cdots\geq m_n$ and $m_1\geq2$. Let $\mm-\e_1:=\{1^{m_1-1},2^{m_2},\ldots,n^{m_n}\}$. By Eq.~\eqref{sig:nq} we have
$$
\frac{\sum_{\pi\in P^{\pm}(\mm)}t^{\des(\pi)}}{(1-t)^{m+1}}=\sum_{k\geq0}f(k)\frac{(2k+m_1)}{m_1}t^k
\quad\text{and} \quad\frac{\sum_{\pi\in P^{\pm}(\mm-\e_1)}t^{\des(\pi)}}{(1-t)^{m}}=\sum_{k\geq0}f(k)t^k,
$$
where $f(k)=\frac{(2k+1)\cdots(2k+m_1)\cdots(2k+1)\cdots(2k+m_{n-1})\cdots(2k+1)\cdots(2k+m_n-1)}{(m_1-1)!\cdots m_{n-1}!(m_n)!}$.
Note that $m>\frac{m_1}{2}$ and $\sum_{\pi\in P^{\pm}(\mm)}t^{\des(\pi)}\in\PF_1[t^i]$. Thus by Lemma~\ref{euler},  the polynomial $\sum_{\pi\in P^{\pm}(\mm-\e_n)}t^{\des(\pi)}$ is a polynomial in $\PF[t^i]$ implies that of $\sum_{\pi\in P^{\pm}(\mm)}t^{\des(\pi)}$. The theorem then follows by induction on $m$.
\end{proof}

\begin{remark}
Note that the above approach is also available for Simion's result~\cite{si} about the real-rootedness of $\sum_{\pi\in P(\mm)}t^{\des(\pi)}$.
\end{remark}

\begin{corollary}
The descent polynomial 
$$
\sum_{\pi\in P^{\pm}(\mm)}t^{\des(\pi)}
$$ 
is log-concave and unimodal for each $\mm\in\PP^n$.
\end{corollary}

\end{document}